\documentclass[11pt]{amsart}
\usepackage[foot]{amsaddr}
\usepackage{amssymb,amsmath,graphicx,array,xcolor}
\usepackage[toc]{appendix}

\newtheorem{theorem}{\sc Theorem}[section]
\newtheorem{corollary}[theorem]{\sc Corollary} 
\newtheorem{lemma}[theorem]{\sc Lemma}
\newtheorem{proposition}[theorem]{\sc Proposition}
\newtheorem{definition}[theorem]{\sc Definition}

\title{Extending Lagrangian Transformations to Nonconvex Scalar Conservation Laws}
\author{Prerona Dutta}
\address{Department of Mathematics, The Ohio State University}
\email{dutta.105@osu.edu}

\begin{document}
\maketitle
\vspace*{-0.5cm}
\begin{abstract}
The present paper studies a method of finding Lagrangian transformations, in the form of particle paths, for all scalar conservation laws having a smooth flux. These are found using the notion of weak diffeomorphisms. More precisely, from any given scalar conservation law, we derive a Temple system having one linearly degenerate and one genuinely nonlinear family. We modify the system to make it strictly hyperbolic and prove an existence result for it. Finally we establish that entropy admissible weak solutions to this system are equivalent to those of the scalar equation. This method also determines the associated weak diffeomorphism.\\


\end{abstract}

\section{Introduction}
Studies have often been conducted on families of conservation laws with the aim of developing a method for finding Lagrangian transformations. This involves describing the problem equivalently by constructing particle paths. Classically, this means that the system can be described as the flow in time of a diffeomorphism of space. In our framework, once the solutions are no longer smooth, the erstwhile diffeomorphisms become invertible bi-Lipschitz mappings, referred to as weak diffeomorphisms, that also satisfy systems of conservation laws.  Through such weak diffeomorphisms, the notion of a particle path can be extended to scalar equations and to systems which do not include velocity fields explicitly. A more exciting problem is to find particle paths for many systems where this had not been observed previously. In particular for conservation laws, we need to find methods that work even when the particle paths are not differentiable and when there is no natural velocity field.\\

Previous work in \cite{HKT} illustrated such a procedure for scalar conservation laws in one space variable, that works when the flux is convex.  Holmes, Keyfitz and Tiglay established in \cite{HKT} that under the standard notions of admissibility of weak solutions, weak solutions of the original equation correspond to weak solutions of the transformed system. For classical solutions, the method in \cite{HKT} is related to Ebin and Marsden's work in \cite{EM} on systems of incompressible fluid flow, where they showed that motions were geodesics on the space of volume preserving diffeomorphisms under a suitable metric. For weak solutions, the authors of \cite{HKT} followed a similar approach considering particle paths as curves through the space of absolutely continuous and invertible isomorphisms, which they called weak diffeomorphisms. Weak diffeomorphisms for compressible gas dynamics equations were also found in \cite{HKT}. As these systems describe fluid flow, existence of particle paths in the form of weak diffeomorphisms followed naturally in this case. \\

A different notion of Lagrangian variables was used by Bianchini and Marconi in \cite{BM1}. They used generalized characteristics to define their notion of Lagrangian representation and applied it to prove results on the structure of bounded entropy admissible weak solutions to scalar conservation laws. The Lagrangian representation in this case was not related to the idea of a weak diffeomorphism. In fact, the construction in \cite{BM1} was based on wavefront tracking. Thereafter, to find an approximation scheme and compactness estimates, the authors of \cite{BM1} utilized the transport collapse method first introduced by Brenier in \cite{Bren}. \\

Our definition of a Lagrangian transformation is different from the perspective in \cite{BM1}. Our goal is to extend the idea of particle paths for scalar conservation laws to weak solutions. The difficulty in this problem lies in the fact that such particle paths cannot be characteristic curves as characteristics may intersect. With this in mind, given a scalar conservation law 
\begin{equation} \label{eq1}
\rho_t+f(\rho)_x=0\,,\quad \rho(x,0)=\rho_0(x)\,,\quad x\in \mathbb R\,
\end{equation}
since an abstract scalar conservation law may not describe the flow of a substance, we define $u := f(\rho)/\rho$ and rewrite (\ref{eq1}) with $u$ defining the flow speed as
\begin{equation} 
\rho_t+(u\rho)_x=0\,,\quad \rho(x,0)=\rho_0(x)\,,\quad x\in \mathbb R\,~.
\end{equation}
Given a particle at $x$ at $t = 0$, define $\gamma(x, t)$ as the position it would reach at time $t$ if it travels with velocity $u$ at every time. Thus $\gamma$ is defined by the equation
\begin{equation} \label{gammadef}
\dot\gamma = u\circ\gamma \quad \textrm{or} \quad \gamma_t(x,t)=u(\gamma(x,t),t)\, \quad \textrm{with}\quad \gamma(x,0)=x\,~.
\end{equation}
Our objective is to establish that for each fixed $t$, $\gamma$ is a diffeomorphism for smooth solutions and generalizes to a weak diffeomorphism in the absence of classical solutions. The authors of \cite{HKT} addressed this for the case when the smooth flux $f$ is convex. In the present paper, we determine how to find weak diffeomorphisms for any scalar conservation law with a smooth flux, not necessarily convex. After some manipulation of the conservation law, we define a mapping $\gamma$ that satisfies a Hamilton-Jacobi equation. From this, as done in \cite{HKT}, we derive a Temple system with one linearly degenerate and one genuinely nonlinear family. The system obtained in \cite{HKT} was strictly hyperbolic owing to the convexity assumption on the flux. Without the convexity of $f$, the resulting Temple system would not be strictly hyperbolic. There are no relevant existence results for non-strictly hyperbolic Temple systems as yet. Our achievement in this paper is to convert the problem (\ref{eq1}) into an equivalent problem whose diffeomorphism equation is a strictly hyperbolic Temple system. Furthermore, we slightly extend a prior result on existence of global large data solutions for strictly hyperbolic Temple systems. In \cite{LT} this was proved by assuming both families to be genuinely nonlinear. In the current paper we develop an existence result along the lines of \cite{LT}, with suitable modifications to account for the case when one family is linearly degenerate. Finally we recover $\gamma$ from the Temple system and show the equivalence of solutions of the system with those of the original equation.\\

This paper is organized as follows. In Section 2, we show that a scalar conservation law with bounded initial data can be transformed to an equivalent Cauchy problem for a convex conservation law. Then we state the existence result (proved in Appendix A) for strictly hyperbolic Temple systems having one linearly degenerate and one genuinely nonlinear family. Finally we prove that weak solutions of the derived Temple system are equivalent to those of the original scalar equation. We refer to \cite{Bressan, Dafermos, Smoller} for definitions and results related to conservation laws that are used throughout the rest of this paper, specifically pg. 81-85 in \cite{Bressan} for definitions of admissibility conditions and entropy admissible weak solution.

\section{Weak Diffeomorphisms for Scalar Conservation Laws}
Given a scalar conservation law of the form (\ref{eq1}) where $f$ is a smooth flux and $\rho_0 \in I = [a,b]$ for any $a,b \in \mathbb{R}$, we rewrite it as
\begin{equation} \label{eq2}
\rho_t+(u\rho)_x=0\,,\quad \rho(x,0)=\rho_0(x)\,,\quad x\in \mathbb R\,
\end{equation}
where we define $u$ to be
\begin{equation} \label{udef1}
u\equiv F(\rho) =\frac{f(\rho)}{\rho}\,.
\end{equation}
The weak diffeomorphism corresponding to (\ref{eq2}) is defined by
\begin{equation} \label{gammadef1}
\dot\gamma = u\circ\gamma \quad \textrm{or} \quad \gamma_t(x,t)=u(\gamma(x,t),t)\, \quad \textrm{with}\quad \gamma(x,0)=x\,~.
\end{equation}
We show that for each fixed $t$, $\gamma$ defines a diffeomorphism of $x$ for smooth solutions and generalizes to a weak diffeomorphism when classical solutions no longer exist. This was achieved for a special case in \cite{HKT} where the authors assumed that $f$ was convex and $\rho, F, F'>0$, to deduce an equation for $\gamma$ using the invertibility of $F$. Motivated by \cite{HKT}, in this paper we provide a construction for weak diffeomorphisms for a Cauchy problem with bounded data and a smooth flux, not necessarily convex.\\

\subsection{Transforming the problem}
Given a scalar conservation law
\begin{equation}\label{maineq}
\rho_t+f(\rho)_x=0\,,\quad \rho(x,0)=\rho_0(x)\,,\quad x\in \mathbb R\,
\end{equation}
where $f$ is any smooth flux,
assume that $\rho_0 \in I = [a,b]$ for some $a,b\in \mathbb{R}$.\\

It is clear from (\ref{udef1}) that for $F$ to be well-defined, $\rho$ must be non-zero. Also, $\rho$ must maintain the same sign over $I$ so that the Riemann problem for the derived Temple system has a solution. Therefore if $I$ contains $0$, we use the following translation.\\

Given $\rho_0 \in I=[a,b]$ with $a<0<b$, choose a constant $L \in \mathbb{R}$ such that $L>a$ and define $\sigma := \rho +L$. Then (\ref{maineq}) becomes
\begin{equation}\label{eqL}
\sigma_t+f(\sigma - L)_x=0\,,\quad \sigma(x,0)=\sigma_0(x)\,,\quad x\in \mathbb{R}\,
\end{equation}
where $\sigma_0 \in \tilde{I}$ with $\tilde{I} = I + L\subset \mathbb{R}^{+}$.\\

From (\ref{eqL}),
\begin{equation}\label{newF}
u \equiv F(\sigma) = \frac{f(\sigma - L)}{\sigma}~.
\end{equation}
The equation for $\gamma$ was derived in \cite{HKT} by assuming that $F$ is invertible. In that case, from (\ref{newF}) taking $\sigma = b(u)$ where $b=F^{-1}$, (\ref{eqL}) becomes
\begin{equation}\label{geqn}
b'(u)u_t+ub'(u)u_x+b(u)u_x = 0~.
\end{equation}
Differentiating (\ref{gammadef1}) with respect to $t$ and substituting (\ref{geqn}),
\begin{equation}\label{geqn2}
\gamma_{tt} = -\left(\frac{b(u)}{b'(u)}\right)u_x \circ \gamma~.
\end{equation}
Differentiating (\ref{gammadef1}) with respect to $x$, 
\begin{equation}\label{ux}
u_x = \left(\gamma_t \circ \gamma^{-1}\right)_x = \left(\frac{\gamma_{xt}}{\gamma_x}\right)\circ \gamma^{-1}~.
\end{equation}
Using (\ref{ux}) to eliminate $u_x$ from (\ref{geqn2}) we obtain
\begin{equation}\label{2.10}
\left(\frac{b'(\gamma_t)}{b(\gamma_t)}\right)\gamma_{tt} = - \frac{\gamma_{xt}}{\gamma_x}~.
\end{equation}
Integrating (\ref{2.10}) from time $0$ to $t$ and using $\gamma_t(x,0) = F(\sigma_0(x))$, $\gamma_x(x,0) = 1$ results in
\begin{equation}
\ln\frac{b(\gamma_t)}{\sigma_0} = -\ln \gamma_x~
\end{equation}
which yields
\begin{equation}\label{geqn3}
b(\gamma_t) = \frac{\sigma_0}{\gamma_x}~.
\end{equation}
Since $b= F^{-1}$, (\ref{geqn3}) gives a Hamilton-Jacobi type equation
\begin{equation} \label{HJ1}
\gamma_t=F\left(\frac{\sigma_0}{\gamma_x}\right)
\end{equation}
which can be replaced by an equivalent conservation law system, with $\eta=\gamma_x$:
\begin{equation} \label{syst1}
\begin{split} \eta_t-F\left(\frac{v}{\eta}\right)_x&=0\,,\quad \eta(x,0)=1\,,\\
 v_t&=0\,,\quad v(x,0)=\sigma_0(x)\,.
 \end{split}
 \end{equation}
Proving that solutions of system \eqref{syst1} are equivalent to those of
the original equation \eqref{maineq} means showing that given an admissible weak solution $(\eta, v)$ to \eqref{syst1}, one can recover the admissible weak solution to \eqref{maineq} with the corresponding data.\\

The system (\ref{syst1}) has the structure of a Temple system, for which the shock and rarefaction waves coincide in state space. In quasilinear form, (\ref{syst1}) becomes  
\begin{equation}\label{temple}
{\begin{pmatrix}
\eta\\v
\end{pmatrix}}_t
+
\begin{pmatrix}
\frac{v}{\eta^2}F'\left(\frac{v}{\eta}\right) & -\frac{1}{\eta}F'\left(\frac{v}{\eta}\right)
\\
0 & 0
\end{pmatrix}
{\begin{pmatrix}
\eta\\v
\end{pmatrix}}_x
= 
\begin{pmatrix}
0\\0
\end{pmatrix}
\end{equation}
with eigenvalues $\lambda_1 = 0$ and $\lambda_2 = \frac{v}{\eta^2}F'\left(\frac{v}{\eta}\right)$.
The family corresponding to the eigenvalue $\lambda_1 = 0$ is linearly degenerate, while the family corresponding to the eigenvalue $\lambda_2$ is genuinely nonlinear. From
(\ref{newF}),
$$F'(\sigma) = \frac{\sigma f'(\sigma-L) - f(\sigma-L)}{\sigma^2}~.$$
Note that a system of the form (\ref{syst1}) associated with the scalar equation \eqref{maineq} via $\gamma$ is strictly hyperbolic only when $F$ is invertible. Existence of global large data solutions for strictly hyperbolic $2\times2$ Temple systems was proved in \cite{LT}. If $F'(\sigma) = 0$ for some $\sigma\in \tilde{I}$, we modify the conservation law in Lemma \ref{transf} to obtain an equation equivalent to (\ref{maineq}) such that the associated Temple system is strictly hyperbolic. The new equation and the corresponding system establish our results. Our process replaces $F$ with a velocity function $G$ where $G' > 0$.

\begin{lemma}\label{transf}
Given $f$ in (\ref{eqL}) such that $F'(\sigma) = 0 $ for some $\sigma \in \tilde{I}$, define $g(\sigma) = f(\sigma-L) + K$ and $G$ by
$$G(\sigma) = \frac{g(\sigma)}{\sigma}$$
where $K \in \mathbb{R}$ is a constant. 
If $K<\min\limits_{\sigma\in \tilde{I}}\{\sigma f'(\sigma - L)-f(\sigma - L)\}$, then $G' > 0$ and $G$ is invertible for $\sigma \in \tilde{I}$. Moreover, the system replacing $F$ in (\ref{syst1}) with $G$ is strictly hyperbolic.
\end{lemma} 
\begin{proof}
Here $G'$ is given by
$$
G'(\sigma) = \frac{\sigma f'(\sigma - L) - f(\sigma - L) - K}{\sigma^2}
$$
for $\sigma \in \tilde{I}$. Thus $G'>0$ if $\sigma f'(\sigma - L) - f(\sigma - L)>K$. We choose any $K<\min\limits_{\sigma\in \tilde{I}}\{\sigma f'(\sigma - L)-f(\sigma - L)\}$.  \\
\quad\\
Now for the modified system 
\begin{equation}\label{newsys}
\begin{split} \eta_t-G\left(\frac{v}{\eta}\right)_x&=0\,,\quad \eta(x,0)=1\,,\\
 v_t&=0\,,\quad v(x,0)=\sigma_0(x)\,,
 \end{split}
\end{equation}
we obtain $\lambda_2 = \frac{v}{\eta^2}G'\left(\frac{v}{\eta}\right) > 0$. Thus (\ref{newsys}) is strictly hyperbolic.
\end{proof}
\quad\\
{\bf Note:} For $\rho_0 \in [a,b]$, if $a>0$, then $L=0$ and $\sigma=\rho+0=\rho$. If $b<0$, then define $\sigma = -\rho$.
\subsection{Main Results}
In the previous section we transformed the Cauchy problem (\ref{maineq}) for a smooth flux $f$  into a strictly hyperbolic Temple system. The modified scalar equation is
\begin{equation}\label{neweq}
\sigma_t+g(\sigma)_x=0\,,\quad \sigma(x,0)=\sigma_0(x)\,,\quad x\in \mathbb R\,
\end{equation}
where $g$ is obtained from $f$ via Lemma \ref{transf} and $\sigma$ is obtained from $\rho$ so that $\sigma>0$ in $\tilde{I}$. Proposition \ref{eqv} shows that (\ref{neweq}) is equivalent to the original scalar equation (\ref{maineq}).\\

To prove the existence of global large data solutions for strictly hyperbolic Temple systems in \cite{LT}, Leveque and Temple assumed both families to be genuinely nonlinear and used Godunov's method. Here we extend their result to when one family is linearly degenerate. 
\begin{theorem}\label{exist}
Consider the Cauchy problem
\begin{equation}\label{newsyst}
\begin{split} \eta_t-G\left(\frac{v}{\eta}\right)_x&=0\,,\quad \eta(x,0)=1\,,\\
 v_t&=0\,,\quad v(x,0)=\sigma_0(x)\,
 \end{split}
\end{equation}
where $\sigma_0\in\tilde{I}$ has bounded total variation, $G\in\mathcal{C}^2$ and $G' > 0$ for $\sigma \in \tilde {I}$. 
Then (\ref{newsyst}) has a unique admissible weak solution $(\eta, v) \in \left(\mathcal{C}(0,\infty), BV^2 \right)$ for $t>0$.
\end{theorem}
\begin{proof}
The proof is given in Appendix A and proceeds as in \cite{LT}, with modifications to account for the linearly degenerate family in system (\ref{newsyst}).
\end{proof}
As a direct consequence of Theorem \ref{exist}, we have the corollary:
\begin{corollary} \label{thref}
Let $(\eta, v) \in \left(\mathcal{C}(0,\infty), BV^2 \right)$ be the admissible weak solution to the system (\ref{newsyst}). The distributional solution $\gamma$ to 
\begin{equation}\label{gammasol}
\gamma_x(x,t) = \eta(x,t)~~,~~\gamma_t(x,t) = G\left(\frac{v(x,t)}{\eta(x,t)}\right)~~,~~\gamma(x,0) = x
\end{equation}
is well-defined, absolutely continuous and invertible and $\gamma^{-1}$ is absolutely continuous. 
\end{corollary}
\begin{proof}
For fixed $t$, $\gamma$ is an antiderivative of a strictly positive function of bounded variation. Therefore it follows using the Fundamental Theorem of Calculus for Lebesgue Integrals (Theorem 3.35 in \cite{Folland}) that $\gamma$ is an absolutely continuous function of $x$ and is invertible.\\

The function $\gamma$ is strictly increasing in $x$ and differentiable almost everywhere by the same theorem. Thus the inverse in $x$ of $\gamma$ is continuous and strictly increasing. Let $X$ be the set of points at which $\gamma$ has a finite, positive derivative and let $Y$ be the collection of the remaining points. So $Y$ is a set of measure zero. As $\gamma$ is absolutely continuous, it satisfies the Lusin N-property  (Lemma 7.25 in \cite{Rudin}), implying that $\gamma(Y)$ is also a set of measure zero. Let $A=\gamma(X)$ and $B=\gamma(Y)$. We know that $\gamma^{-1}$ has finite positive derivative at every point in $A$ and $B$ is a set of measure zero. Thus $\gamma^{-1}$ maps any measure zero subset of $A$ to a set of measure zero and also maps every subset of $B$ to a set of measure zero, i.e., a subset of Y. Therefore $\gamma^{-1}$ satisfies the Lusin N-property as well, which shows that it is absolutely continuous.
\end{proof}
\quad\\
This leads to our main result, where we prove that admissible weak solutions to (\ref{neweq}) can be recovered from those of (\ref{newsyst}).

\begin{theorem}\label{mainthm}
Define
\begin{equation}\label{soln}
\sigma := \frac{v\left(\gamma^{-1}(x,t),t\right)}{\eta\left(\gamma^{-1}(x,t),t\right)}
\end{equation}
where $(\eta, v)$ is the admissible weak solution to (\ref{newsyst}) and $\gamma$ is defined in Corollary \ref{thref}. Then $\sigma$ is an admissible weak solution to (\ref{neweq}).
\end{theorem}
\begin{proof} With $\gamma^{-1}$ defined in Corollary \ref{thref}, construct $\sigma$ as in (\ref{soln}). Let $\varphi(x,t)$ be a test function, i.e. $\varphi\in\mathcal{C}^1$ and has compact support. Recalling $g(\sigma) = \sigma G(\sigma)$, we have
\begin{align*}
I &:= \iint\left(\varphi_t\sigma + \varphi_x g(\sigma)\right)~dxdt\\
&= \iint \left(\varphi_t\left(\frac{v}{\eta}\right)\circ \gamma^{-1} + \varphi_x\cdot\left(\frac{v}{\eta}G\left(\frac{v}{\eta}\right)\right)\circ \gamma^{-1}\right)~dxdt~.
\end{align*}
Using the change of variable $(x,t)\mapsto(\gamma(x,t),t)$
\begin{equation}\label{istep}
I = \iint \left(\left(\varphi_t\circ \gamma\right)\left(\frac{v}{\eta}\right)+\left(\varphi_x\circ \gamma\right)\left(\frac{v}{\eta}G\left(\frac{v}{\eta}\right)\right)\right)\eta~dxdt~.
\end{equation}
Now,
\begin{equation}\label{par1}
\frac{\partial}{\partial t}(\varphi(\gamma(x,t),t)) = \varphi_x(\gamma,t)\gamma_t + \varphi_t(\gamma,t)~,
\end{equation}
\begin{equation}\label{par2}
\frac{\partial}{\partial x}(\varphi(\gamma(x,t),t)) = \varphi_x(\gamma,t)\gamma_x~.
\end{equation}
Let $\tilde{\varphi}(x,t) = \varphi(\gamma(x,t),t)$. From (\ref{par1})-(\ref{par2}) we get
$$\varphi_t\circ \gamma = \tilde{\varphi}_t - \frac{\tilde{\varphi}_x}{\gamma_x}\gamma_t~~~\text{and}~~~\varphi_x\circ \gamma = \frac{\tilde{\varphi}_x}{\gamma_x}~.$$
Then (\ref{istep}) becomes
\begin{align}\label{cov}
\begin{split}
I
 &= \iint\left(\tilde{\varphi}_tv + \tilde{\varphi}_x\left(\frac{v}{\eta}G\left(\frac{v}{\eta}\right) - \frac{v}{\eta}G\left(\frac{v}{\eta}\right)\right)\right)~dxdt\\
 &= \iint \tilde{\varphi}_tv~dxdt\\
 &= -\int \tilde{\varphi}(x,0)\sigma_0(x)~dx~.
 \end{split}
\end{align}
Thus $\sigma$ as defined in (\ref{soln}) using the weak solution to the system (\ref{newsyst}), satisfies the definition of a weak solution to the equation (\ref{neweq}).\\

Next we establish that admissibility conditions for the system (Section 4.4 in \cite{Bressan}) correspond to admissibility conditions for the scalar conservation law. In order to do so, we prove that given a convex extension for (\ref{neweq}), we can recover a convex extension for (\ref{newsyst}) and vice versa. 

\begin{definition}
A pair of $\mathcal{C}^1$ functions $(\mathcal{E}, \mathcal{Q}):\mathbb{R} \to\mathbb{R}$ is an entropy-entropy flux pair for (\ref{neweq}) if
\begin{equation}\label{EEF}
\mathcal{Q}'(\sigma)~=~\mathcal{E}'(\sigma)\cdot g'(\sigma)
\end{equation}
at every $\sigma$ where $\mathcal{E}, \mathcal{Q}$ and $f$ are differentiable. In particular, if $\mathcal{E}$ is convex and smooth, then the pair $(\mathcal{E}, \mathcal{Q})$ is called a {\bf convex extension}.
\end{definition}

A convex extension $(\bar{\mathcal{E}}, \bar{\mathcal{Q}})$ of (\ref{newsyst}) (Definition 4.4 in \cite{Bressan}) satisfies the equation $\bar{\mathcal{E}}_t+ \bar{\mathcal{Q}}_x = 0$ for smooth variables. From this, using the chain rule and (\ref{newsyst}) we obtain
\begin{equation*}
\bar{\mathcal{Q}}_{\eta} = \bar{\mathcal{E}}_{\eta}\cdot\frac{v}{\eta^2}G'\left(\frac{v}{\eta}\right)~~\text{and}~~
\bar{\mathcal{Q}}_{v} = -\bar{\mathcal{E}}_{\eta}\cdot\frac{1}{\eta}G'\left(\frac{v}{\eta}\right)~.
\end{equation*}
Equating second order mixed partial derivatives yields $\bar{\mathcal{E}}_{\eta\eta} + \frac{v}{\eta}\bar{\mathcal{E}}_{\eta v}=0$ and solving this partial differential equation we get
\begin{align}\label{conv}
\begin{split}
\bar{\mathcal{E}}\left( \frac{v}{\eta} \right) &= \eta\cdot S\left( \frac{v}{\eta} \right),\\
\bar{\mathcal{Q}}\left( \frac{v}{\eta} \right) &= - S\left(\frac{v}{\eta}\right)G\left(\frac{v}{\eta}\right) + \int^{v/\eta} S'(x)G(x)~dx
\end{split}
\end{align}
where $\bar{\mathcal{E}}$ is convex if $S''>0$.\\

Now we return to the proof of Theorem \ref{mainthm}. Suppose that $(\mathcal{E}_1, \mathcal{Q}_1)$ is a convex extension for the scalar equation (\ref{neweq}), i.e., $\mathcal{Q}_1'(\sigma) = \mathcal{E}_1'(\sigma)g'(\sigma)$. Then $\sigma$ is an admissible solution to (\ref{neweq}), if for any test function $\psi$,
\begin{equation}\label{adm}
J:= \iint\left(\psi_t \mathcal{E}_1 + \psi_x \mathcal{Q}_1\right)~dxdt~\geq~0~.
\end{equation}
Now we use a change of variable $(x,t)\mapsto(\gamma(x,t),t)$ and rewrite $J$ as
\begin{equation}\label{adm1}
J = \iint \left(\psi_t (\gamma, t)\mathcal{E}_1 \left(\frac{v}{\eta}\right)+ \psi_x (\gamma, t)\mathcal{Q}_1\left(\frac{v}{\eta}\right)\right)\eta~dxdt.
\end{equation}
Assuming (\ref{adm}), applying (\ref{par1})-(\ref{par2}) to (\ref{adm1}) and using the notation $\tilde{\psi}(x,t) = \psi(\gamma(x,t),t)$, we proceed as in (\ref{cov}) to deduce that
\begin{equation}
J:= \iint \left(\tilde{\psi}_t \mathcal{E}_2(\eta, v) + \tilde{\psi}_x \mathcal{Q}_2(\eta, v)\right)~dxdt~\geq~0~,
\end{equation}
where $\mathcal{E}_2(\eta, v) = \eta \mathcal{E}_1 \left(\frac{v}{\eta}\right)$ and $\mathcal{Q}_2(\eta, v) = \mathcal{Q}_1\left(\frac{v}{\eta}\right) - G\left(\frac{v}{\eta}\right)\mathcal{E}_1\left(\frac{v}{\eta}\right)$. Thus $\mathcal{E}_2$ is of the form obtained in (\ref{conv}) and $(\mathcal{E}_2)_t + (\mathcal{Q}_2)_x = 0$ whenever $(\eta, v)$ is a classical solution to the system (\ref{newsyst}). Furthermore, as $\mathcal{E}_1$ is convex, it follows by a direct computation that the Hessian matrix of $\mathcal{E}_2$ is positive semi-definite. For the reverse implication, consider a convex entropy $\bar{\mathcal{E}}$ for the system (\ref{newsyst}) as defined in (\ref{conv}). Then the convex function $S$ in (\ref{conv}) is a convex entropy for (\ref{neweq}). 
\medskip

This completes the proof.\\
\end{proof}

In conclusion, we observe that $\rho$ is an admissible weak solution to the Cauchy problem for the scalar conservation law (\ref{maineq}) where $\rho = \sigma - L$ and $\sigma$ is given by (\ref{soln}).

\begin{proposition}\label{eqv}
Equations (\ref{maineq}) and (\ref{neweq}) are equivalent in the sense that entropy admissible weak solutions of  (\ref{neweq}) are equivalent to those of (\ref{maineq}).
\end{proposition}
\begin{proof}
The equation (\ref{neweq}) is obtained from (\ref{maineq}) via Lemma 2.1. Given a convex extension $(\mathcal{E}_1, \mathcal{Q}_1)$ for (\ref{maineq}), it satisfies $\mathcal{Q}_1'(\rho) = \mathcal{E}_1'(\rho)f'(\rho)$ for $\rho\in I$. As $\rho$ is an admissible weak solution to (\ref{maineq}), for any test function $\varphi$,
\begin{equation*}
\iint \left(\varphi_t \mathcal{E}_1(\rho) + \varphi_x \mathcal{Q}_1(\rho)\right)~dxdt~\geq~0
\end{equation*}
which means
\begin{equation*}
\iint \left(\varphi_t \mathcal{E}_1(\sigma - L) + \varphi_x \mathcal{Q}_1(\sigma - L)\right)~dxdt~\geq~0~.
\end{equation*}
Taking $\mathcal{E}_2(\sigma) = \mathcal{E}_1(\sigma - L)$ and $\mathcal{Q}_2(\sigma) = \mathcal{Q}_1(\sigma - L)$ yields 
\begin{equation*}
\iint \left(\varphi_t \mathcal{E}_2(\sigma) + \varphi_x \mathcal{Q}_2(\sigma)\right)~dxdt~\geq~0
\end{equation*}
for any test function $\varphi$, where $\mathcal{Q}_2'(\sigma) = \mathcal{E}_2'(\sigma)g'(\sigma)$ and $\sigma$ is an admissible weak solution to (\ref{neweq}). Therefore $(\mathcal{E}_2, \mathcal{Q}_2)$ is a convex extension for (\ref{neweq}) obtained from a convex extension of (\ref{maineq}) and the reverse follows similarly. This implies entropy admissible weak solutions of (\ref{maineq}) and (\ref{neweq}) are equivalent.
\end{proof}

\appendix
\section*{Appendix A}
\renewcommand{\theequation}{A.\arabic{equation}}
In this appendix, we provide details for the proof of Theorem 2.2 based on results and methods from \cite{HLL, LW, LT}.\\

The following procedure uses Godunov's method for approximating hyperbolic conservation laws by finite differences. This method chooses a piecewise constant approximation of the data and solves a Riemann problem at each discontinuity. The state variables are then averaged over each space interval. To define a piecewise constant approximation at the next time step requires that waves emerging from one inter-cell boundary not interact with waves created at adjacent boundaries. Controlling the Courant number as shown in (\ref{courant}) avoids such wave interactions.\\

The system (\ref{newsyst}) is a Temple system, for which shock and rarefaction curves are straight lines that coincide in $\eta-v$ space. For systems having two genuinely nonlinear families, \cite{LT} shows the existence of an invariant region. However contact discontinuity curves may not be straight lines and the existence of an invariant region is unclear. In fact, the shape of contact discontinuity curves in $\eta-v$ space depends on the function $G(\eta, v)$.\\

In our problem, $G$ is a function of $\frac{v}{\eta}$ and the right eigenvectors corresponding to the eigenvalues $\lambda_1 = 0$ and $\lambda_2 = \frac{v}{\eta^2}G'\left(\frac{v}{\eta}\right)$ are
$${\bf r_1} = \begin{pmatrix}
1\\\frac{v}{\eta}
\end{pmatrix} ~~\text{and}~~
{\bf r_2} = \begin{pmatrix}
1\\0
\end{pmatrix} ~.$$
Thus the contact discontinuity curves are rays from the origin. From the given initial data in (\ref{newsyst}), we see that there is a potential invariant region in the form of a quadrilateral {\bf Q} whose vertices are given by $(m/M, m),(1, m), (M/m, M) $ and $(1, M)$, where $0<m\leq\sigma_0\leq M$. The sides of {\bf Q} are either contact discontinuity curves corresponding to $\lambda_1$, i.e., rays from the origin, or are horizontal lines in the $\eta-v$ plane with $\lambda_2$ decreasing as $\eta$ increases. In Claim 1 below we establish that {\bf Q} is indeed an invariant region for our problem. Hence, we can use the arguments related to Godunov's method from \cite{LT} that depend on having an invariant region in the form of a quadrilateral, with alterations to accommodate the fact that one family is linearly degenerate and the other is genuinely nonlinear.\\

For our problem, the 1-Riemann invariant is given by $\frac{v}{\eta}$ and the 2-Riemann invariant is $v$, with $i$-th Riemann invariant for $i=1,2$ defined as in \cite{LT}. We denote these by $p$ and $q$ respectively. Let $w=(\eta,v)$ and $P=\left(-G\left(\frac{v}{\eta}\right), 0\right)$. The Riemann problem for (\ref{newsyst}) is
\begin{align} \label{RP}
w_t + P(w)_x &= 0\\
w(x,0)&=w^0(x) = \label{ICapp}
\begin{cases}
w_L,~~x\leq 0\\
w_R,~~x>0
\end{cases}
~.
\end{align}
To apply Godunov's method, take $h$ to be the mesh length in $x$ and $k$ to be a time step such that 
\begin{equation}\label{courant}
\frac{k}{h} \sup\limits_{w\in{\bf Q}} |\lambda_i(w)|~<~1~.
\end{equation}
For $j\in\mathbb{Z}$ and $n\in\mathbb{N}$, let $x_j = jh$ and $t_n = nk$. Then the approximation to the initial condition is defined at mesh points by
\begin{equation}\label{pieceIC}
w_j^0 = \frac{1}{h}\int_{x_{j-1}}^{x_j}w^0(x) dx~.
\end{equation}
Recall that {\bf Q} is a quadrilateral containing $w_L$ and $w_R$. In $(p,q)$ coordinates {\bf Q} is a rectangle. So if $w^0(x)\in {\bf Q}$ for all $x$, $w_j^0$ lies in {\bf Q} for all $j$. For every $j$, let $\tilde{w}(x, t_{n}) = w_{j}^{n}$ denote the approximate solution at time $t_{n}$ for $x\in(x_{j-1}, x_j]$. The piecewise constant data at each $t_n$ give a Riemann problem at every $x_j$. Solving the local Riemann problem at each $x_j$, let $\tilde{w}(x, t_{n+1})$ denote the solution in $[t_{n}, t_{n+1})$. By (\ref{courant}), the waves that emanate from $(x_j, t_n)$ do not intersect before the next time step $t_{n+1}$. We average the local solution in each interval by defining
\begin{equation}\label{wjdef}
w_j^{n+1} = \frac{1}{h}\int_{x_{j-1}}^{x_j}\tilde{w}(x, t_{n+1}) dx~
\end{equation}
to approximate $w$ at $t_{n+1}$. 
As one family has characteristic speed 0, there is a discontinuity at each $x_j$. Define the extension $w_h$ of the grid function as 
\begin{equation}\label{whdef}
w_h(x,t) = w_j^{n+1}~~\text{for}~~ (x,t)\in (x_{j-1}, x_j] \times[t_{n}, t_{n+1})~.
\end{equation}
Our aim is to show that given a sequence of mesh lengths $\{h_l\}$ converging to $0$, there is a subsequence of $\{h_l\}$ for which $w_h(x,t)$ converges boundedly a.e. to a function $w(x,t)$ and that $w$ is a weak solution to the Riemann problem (\ref{RP}).\\
\quad\\
The following claims lead to our result.\\
\quad\\
{\bf Claim 1:} {\bf Q} is an invariant region for the Riemann problem (\ref{RP}): if $w^0 \in {\bf Q}$ then $w_j^n \in {\bf Q}$ for all $j, n$.
\begin{proof}
We know that {\bf Q} is a quadrilateral and $w_L, w_R\in {\bf Q}$. The solution to the Riemann problem always consists of a contact discontinuity with speed 0 and a positive speed shock or rarefaction, with $p$ constant on the contact discontinuity and $q$ constant on the other wave. In $(p,q)$ coordinates {\bf Q} is a rectangle. Given $n$, if $w_j^{n}$ is in the rectangle {\bf Q} for all $j$, ${\bf Q}$ contains $\tilde{w}(x, t_{n+1})$ for $x \in [x_{j-1}, x_j)$. This implies $w_j^{n+1}\in {\bf Q}$ as it is the average of all $\tilde{w}(x, t_{n+1})$ over $[x_{j-1}, x_j)$ at $t_{n+1}$. Thus, the claim follows by induction.
\end{proof}
Define distance between two points as
\begin{equation}\label{norm}
\|w_L - w_R\| = |p(w_L) - p(w_R)| + |q(w_L) - q(w_R)|~.
\end{equation}
\quad\\
{\bf Claim 2:} For all $j,n$, the following hold:
\begin{enumerate}
\item $\|w_{j-1}^{n} - w_j^{n+1}\| + \|w_j^{n+1} - w_j^{n}\|~=~\|w_j^{n} - w_{j-1}^{n}\|$.
\item $
\sum\limits_j \|w_j^{n+1}-w_{j-1}^{n+1}\|~\leq~\sum\limits_j \|w_j^{n}-w_{j-1}^{n}\|
$.
\end{enumerate}
\begin{proof}  \begin{enumerate}
\item  For any $j$ and $n$, $p(w_{j-1}^{n}) -  p(w_{j}^{n}) = p(w_{j-1}^{n}) - p(w_j^{n+1}) + p(w_j^{n+1}) -  p(w_{j}^{n})$. In $(p,q)$ coordinates, $w_j^{n+1}$ is the intersection of a horizontal line from $w_{j-1}^{n}$ and a vertical line from $w_{j}^{n}$. Note that $w_j^{n+1}$ is obtained by averaging the local solution in $[x_{j-1}, x_j)$ at $t_{n+1}$. Therefore both its coordinates lie between the coordinates of $w_{j-1}^{n}$ and $w_{j}^{n}$, so $|p(w_{j-1}^{n}) - p(w_j^{n+1})|$ and $|p(w_j^{n+1}) -  p(w_{j}^{n})|$ have the same sign. Thus,
\begin{equation*}
|p(w_{j-1}^{n}) - p(w_j^{n+1})| + |p(w_j^{n+1}) -  p(w_{j}^{n})| = |p(w_{j-1}^{n}) -  p(w_{j}^{n})|
\end{equation*}
and a similar equality holds for $q$. Hence, (1) follows using (\ref{norm}).\\
\item By the triangle inequality and part 1,
\begin{align}\label{tvdec}
\begin{split}
\sum\limits_j \|w_j^{n+1}-w_{j-1}^{n+1}\|~&\leq~\sum\limits_j \left(\|w_j^{n+1}-w_{j-1}^{n}\|+\|w_{j-1}^{n}-w_{j-1}^{n+1}\|\right)\\&= \sum\limits_j \left(\|w_j^{n+1}-w_{j-1}^{n}\|+\|w_{j}^{n}-w_{j}^{n+1}\|\right)\\
&= \sum\limits_j \|w_j^{n} - w_{j-1}^{n}\|~.
\end{split}
\end{align}
\end{enumerate}
\end{proof}
Integrating (\ref{RP}) over the rectangle $(x_{j-1}, x_j] \times [t_n, t_{n+1})$ and using the divergence theorem, we get  
\begin{equation}\label{findiff0}
w_j^{n+1} = w_j^{n} + \frac{k}{h}\left[P(w_{j}^{n}) -P(w_{j-1}^{n})\right]~.
\end{equation}
\quad\\
\noindent{\bf Claim 3:} Given $0<T_1<T_2$, there exists a constant $C>0$ such that
\begin{equation}\label{TV}
\int_{-\infty}^{\infty}|w_h(x, T_2) - w_h(x, T_1)|~dx ~\leq~ C|T_2 - T_1| \cdot TV(w^0)
\end{equation}
where $TV(w^0)$ is the total variation of initial data $w^0$ measured in norm (\ref{norm}).
\begin{proof}
Let $0<T_1<T_2$ and $r, R \in \mathbb{Z}$ be such that
$$t_r\leq T_1<t_{r+1}\leq t_R\leq T_2<t_{R+1}~. $$
From this relation, using (\ref{whdef}) we have
\begin{equation}\label{tvclaim}
\int_{-\infty}^{\infty}|w_h(x, T_2) - w_h(x, T_1)|~dx = h\sum\limits_{j=-\infty}^{\infty} |w_j^{R+1} - w_j^{r+1}| \leq h\sum\limits_{j=-\infty}^{\infty} \sum\limits_{n=r+1}^{R} |w_j^{n+1} - w_j^{n}|~.
\end{equation}
Using (\ref{findiff0}) in (\ref{tvclaim}) we get
\begin{equation}\label{a40}
\int_{-\infty}^{\infty}|w_h(x, T_2) - w_h(x, T_1)|~dx \leq k \sum\limits_{j=-\infty}^{\infty} \sum\limits_{n=r+1}^R |{P}(w_{j}^{n}, w_{j+1}^{n}) - {P}(w_{j-1}^{n}, w_{j}^{n})|~.
\end{equation}
As $P$ is Lipschitz,
\begin{equation}\label{tvest}
\sum\limits_{j=-\infty}^{\infty} |{P}(w_{j}^{n}, w_{j+1}^{n}) - {P}(w_{j-1}^{n}, w_{j}^{n})| \leq C\cdot TV(w_h)
\end{equation}
where $C$ is the Lipschitz constant for $P$. (\ref{tvdec}) implies that $\{w_h\}$ has nonincreasing total variation, so we have
\begin{equation}\label{tvest}
\sum\limits_{j=-\infty}^{\infty} |{P}(w_{j}^{n}, w_{j+1}^{n}) - {P}(w_{j-1}^{n}, w_{j}^{n})| ~\leq~C\cdot TV(w^{0})~.
\end{equation}
As $k(R-r)\leq |T_2 - T_1|$, (\ref{a40}) and (\ref{tvest}) together yield (\ref{TV}).
\end{proof}
\quad\\
Claims 2 and 3 show that $w_h$ is continuous in time and has bounded total variation for every $h$. Moreover, $\{w_h\}$ is uniformly bounded on $[0, T]$ for $T<\infty$. By Helly's theorem (see Theorem 1.3 in \cite{MO}), there exists a subsequence of $\{w_h\}$ converging pointwise to a limit $w(x,t)$. In addition, as $\{w_h\}$ is uniformly bounded, this subsequence is said to converge boundedly a.e. to $w(x,t)$. Hence, by Lax and Wendroff's theorem in Section 1 of \cite{LW}, $w(x,t)$ is a weak solution to (\ref{RP}).\\

In order to establish that $w(x,t)$ is an admissible weak solution, let us take a convex extension $(\mathcal{E}, \mathcal{Q})$ for (\ref{RP}) and denote
$$\mathcal{E}_j^{n+1} = \mathcal{E}(w_j^{n+1})~~\text{and}~~ \bar{\mathcal{Q}}(w_{j-1}^{n+1}, w_j^{n+1})=\mathcal{Q}(\tilde{w}(x_{j-1},t_{n+2}))~.$$
Recalling (\ref{wjdef}) and using Jensen's inequality, we have
\begin{equation}\label{jensen}
\mathcal{E}_j^{n+1} =  \mathcal{E}\left(\frac{1}{h}\int_{x_{j-1}}^{x_j}\tilde{w}(x, t_{n+1}) dx\right)\leq \frac{1}{h}\int_{x_{j-1}}^{x_j} \mathcal{E}(\tilde{w}(x, t_{n+1})) dx~.
\end{equation}
Given a convex extension $(\mathcal{E}, \mathcal{Q})$, for a weak solution $w$ to (\ref{RP}) to be admissible, it must satisfy
\begin{equation}\label{hartlax}
\mathcal{E}(w)_t + \mathcal{Q}(w)_x \leq 0~.
\end{equation}
In other words, for any smooth nonnegative test function $\psi$ it must hold that
\begin{multline}\label{lax}
-\int_0^{\infty} \int_{-\infty}^{\infty}\left(\mathcal{E}(w(x,t))\psi_t(x,t) + \mathcal{Q}(w(x,t))\psi_x(x,t)\right)~dxdt \\ - \int_{-\infty}^{\infty}\mathcal{E}(w(x,0))\psi(x,0)~dx~\leq~0~.
\end{multline}
Since (\ref{hartlax}) holds for the solution to every local Riemann problem arising at each $x_j$, integrating (\ref{hartlax}) over $(x_{j-1}, x_j] \times [t_n, t_{n+1})$ we get
\begin{multline*}
\frac{1}{h}\int_{x_{j-1}}^{x_j} \mathcal{E}(\tilde{w}(x, t_{n+1})) dx ~\leq~ \frac{1}{h}\int_{x_{j-1}}^{x_j} \mathcal{E}(\tilde{w}(x, t_{n})) dx~ \\-~ \frac{1}{h}\int_{t_{n}}^{t_{n+1}}\left(\mathcal{Q}(\tilde{w}(x_{j},t)) - \mathcal{Q}(\tilde{w}(x_{j-1},t))\right) dt,
\end{multline*}
which implies
\begin{equation}\label{ent}
\mathcal{E}_j^{n+1} \leq \mathcal{E}_j^{n} - \frac{k}{h}\left(\bar{\mathcal{Q}}(w_{j}^{n}, w_{j+1}^{n}) - \bar{\mathcal{Q}}(w_{j-1}^{n}, w_{j}^{n})\right)~.
\end{equation}
Then by Theorem 1.1 in \cite{HLL}, a weak solution $w$ to (\ref{RP}) obtained using Godunov's method satisfies the entropy condition (\ref{lax}). Thus $w(x,t)$ is an admissible weak solution to (\ref{RP}). Moreover, since $w(x,t)$ has bounded total variation, by Theorem 9.4 in \cite{Bressan} it is the unique entropy admissible weak solution to the Cauchy problem (\ref{RP}).


\end{document}